\documentclass[a4paper]{amsart}%
\usepackage{moreverb}
\usepackage[authoryear,square]{natbib}
\usepackage{amsfonts}
\usepackage{amsmath}
\usepackage{graphicx}
\usepackage{enumitem}
\usepackage{siunitx}
\usepackage{tikz}
\usepackage{pgfplots}
\usepackage{booktabs}
\pgfplotsset{compat=newest}
\usetikzlibrary{calc}

\let\Gamma\varGamma
\let\Theta\varTheta
\let\Lambda\varLambda
\let\Xi\varXi
\let\Pi\varPi
\let\Sigma\varSigma
\let\Upsilon\varUpsilon
\let\Phi\varPhi
\let\Psi\varPsi
\let\Omega\varOmega

\newcommand{\R}{\mathbb{ R}}
\newcommand{\N}{\mathcal N}%

\providecommand{\cont}[1]{\operatorname C^{#1}}
\providecommand{\leb}[1]{\operatorname L^{#1}}
\providecommand{\sobh}[2][]{\operatorname H^{#2}\ifx|#1|\else\if#1D_\Dir\else_{#1}\fi\fi}

\newcommand{\mLL}{{\leb2(\Omega;\leb2(Y))}}
\newcommand{\MLL}{{\leb2(\Omega)}}
\newcommand{\mH}{{\sobh1(\Omega;\sobh1(Y))}}

\newcommand{\MH}{{\sobh1(\Omega)}}

\newcommand{\B}{\mathcal{ B}}
\newcommand{\dealii}{\textsf{deal.II}}

\newcommand{\K}{\mathcal{ K}}
\newcommand{\comp}{\circ}
\newcommand{\sys}{\eqref{eq:macro-2xmicro-system}}
\newcommand{\manusys}{\eqref{eq:manu_pde_formulation}}

\renewcommand{\div}{\operatorname{div}}

\renewcommand{\vec}{\mathbf}
\providecommand{\Gammax}[2][x]{\ensuremath{\Gamma_{#1}^{\text{#2}}}}
\providecommand{\Neu}{\mathrm{N}}
\providecommand{\Dir}{\mathrm{D}}
\providecommand{\dx}{\d x}
\providecommand{\dy}{\d y}

\swapnumbers
\newtheorem{theorem}[subsection]{Theorem}
\newtheorem{definition}[subsection]{Definition}

\title{Parallel two-scale finite element implementation of a system with varying microstructures}

\usepackage{amsopn}

\ifpdf
  \DeclareGraphicsExtensions{.eps,.pdf,.png,.jpg}
\else
  \DeclareGraphicsExtensions{.eps}
\fi
\renewcommand{\d}{\operatorname d\!}

\newcommand\BibTeX{{\rmfamily B\kern-.05em \textsc{i\kern-.025em b}\kern-.08em
T\kern-.1667em\lower.7ex\hbox{E}\kern-.125emX}}

\newcommand{\AM}{\textcolor{black}}%
\newcommand{\CV}{\textcolor{black}}%
\newcommand{\OLak}{\textcolor{black}}%

\newcommand{\D}{\operatorname{D}}
\renewcommand{\vec}[1]{\boldsymbol{#1}}

\usepackage{doi}
\usepackage{hyperref}
\usepackage{cleveref}
\providecommand{\d}{}
\renewcommand{\d}{\operatorname d}

\title[Parallel 2-scale FEM with varying microstructure]{Parallel two-scale finite element implementation of a system with varying microstructure}

\author[O. Lakkis]{Omar Lakkis}
\address{O.L. and C.V., Department of Mathematics, University of Sussex, England, UK}
\email{lakkis.o.maths@gmail.com}
\email{C.Venkataraman@sussex.ac.uk}

\author[A. Muntean]{Adrian Muntean}
\address{A.M., Department of Mathematics and Computer Science, Karlstad University, SE}
\email{adrian.muntean@kau.se}

\author[O. Richardson]{Omar Richardson}
\address{O.R., Simula Consulting AS, Oslo, NO}
\email{omar.richardson@gmail.com}

\author[C. Venkataraman]{Chandrasekhar Venkataraman}

\makeatother
\begin{document}
\maketitle

\begin{abstract}
  We propose a two-scale finite element method designed for
  heterogeneous microstructures.  Our approach exploits domain
  diffeomorphisms between the microscopic structures to gain
  computational efficiency.  By using a conveniently constructed
  pullback operator, we are able to model the different microscopic
  domains as macroscopically dependent deformations of a reference
  domain.  This allows for a relatively simple finite element framework
  to approximate the underlying system of partial differential equations
  with a parallel computational structure.  We apply this technique to a
  model problem where we focus on transport in plant tissues.  We
  illustrate the accuracy of the implementation with convergence
  benchmarks and show satisfactory parallelization speed-ups.  We
  further highlight the effect of the heterogeneous microscopic
  structure on the output of the two-scale systems.  Our implementation
  (publicly available on GitHub) builds on the \dealii{} FEM library.
  Application of this technique allows for an increased capacity of
  microscopic detail in multiscale modeling, while keeping running costs
  manageable.
\end{abstract}

\section{Introduction}
Models involving transport and diffusion phenomena interacting at
multiple scales (multiscale) are ubiquitous in the natural sciences
and engineering \cite{weinan11}.  Multiscale modeling is a key tool
for developing effective techniques to describe these phenomena, which
may otherwise be computationally intractable.  Specifically, assuming
scale-separation in models allows us to examine the interplay between
processes active on vastly different length and time scales; defining,
for example, phenomena on \emph{macroscales} and \emph{microscales}
\cite{horstemeyer90,steinhauser2017computational,Simpson2018,
  engquist2009multiscale, Seguin2020}. In most practically relevant
cases, the microscales are active in the sense that they are hosting
localized phase transitions described mathematically by moving
interfaces with \emph{a priori} known or unknown velocities. \AM{The
  case of known interface velocities is both mathematically and
  computationally well-understood (cf. e.g. \cite{Eden2017}), while
  the case of unknown interface velocities is still a matter of
  concern; compare \cite{Peter2022,Gahn2023, Wiedemann2023,eden2024}
  for some very recent asymptotic analysis results concerning closely
  related reaction-diffusion scenarios which arise in the context of
  reactive transport in porous media.}

Here, we consider a flow process that takes place on two distinct
physical scales.  In the simplest setting, we can identify a
macroscopic scale where a model governs the behavior of a
(macroscopic) fluid in which averages over the microscopic scale enter
as source terms, as well as a microscopic scale that includes
diffusion-based transport and chemical reactions on heterogeneous
domains where the macroscopic fluid concentration plays the role of a
parameter.  We propose a computational framework to model the effect
of heterogeneous microscopic geometries on microscopic and macroscopic
solution profiles, where the species involved satisfy a system of
partial differential equations (PDEs).  Resolving the multiscale
structure of the PDEs, especially the heterogeneous microstructure,
requires careful consideration to avoid computational overload yet
make sure the required accuracy on the macroscale is achieved.  We
present an efficient and robust parallel two-scale finite element
method for the approximation of such systems.  Our focus in the
present work is the consideration of linear systems of equations that
preserve the main challenges that arise from the multiscale nature of
the systems whilst sidestepping the added technical complications that
arise in the nonlinear systems of more interest in applications. We
describe in detail practical aspects related to the implementation,
such as visualization of computational results and we report on
parallelization of the algorithm.  We also showcase the effect of
microstructure heterogeneity by presenting simulations of a simple
linear model for transport with heterogeneous
microstructures. \CV{Although the focus is on linear elliptic systems,
  such systems typically arise after linearisation in the resolution
  of nonlinear parabolic or elliptic systems which are relevant in
  applications, hence this work and the associated algorithm and code
  provides a useful contribution to the numerical solution of
  nonlinear problems.}

As an example model to test the numerical method, we propose a system
of PDEs describing the steady-state profile of transport processes in
plants.  The transport of many substances such as water, hormones and
nutrients, occurs through cell scale processes.  An important and
widely studied example is that of Auxin transport
\cite{twycross2010stochastic}, where active influx and efflux carriers
on cell membranes allow for polarity of the transport process.
Establishing and maintaining precise concentration profiles of Auxin
in plant tissues is crucial in many developmental processes, making
the problem inherently multiscale.  Plant cells are orders of
magnitude smaller than the size of the plant itself. Moreover, these
cells vary greatly depending on plant age and location relative to the
center of the stem.  Taking these considerations into account can
greatly complicate the process of developing a reliable computational
model.  In particular, using a single-scale model to resolve plant
geometries in detail is extremely challenging from a computational
point of view \cite{twycross2010stochastic}. A number of recent
studies focus on multiscale modeling with contributions ranging from
modeling growing process \cite[e.g.]{lockhart1965analysis,
  jensen2015multiscale}) to topological constructions
\cite[e.g.]{godin1998multiscale}.  In \cite{raats07} the author
proposes a system of equations that models the consumption of water by
plant roots and in
\cite{allen2020mathematical,chavarria2010homogenization} the authors
consider multiscale modeling of Auxin transport. The framework we
develop is in fact of relevance to all of these applications as our
simulations of the linear model for transport in plant tissues show in
\Cref{sec:experiments}.

Our main aim in the present work is to provide researchers and
practitioners with a computational tool that helps make
scale-separated problems with varying microstructures computationally
tractable. Our focus is on the simplest relevant models in order to
elucidate the underlying principles whilst avoiding unnecessary
technicalities.  However, the techniques presented in this manuscript
are not limited to the simple models considered here. They can be
flexibly extended to more complex multiscale problems, involving
different kinds of physics such as coupling mechanics and fluid
transport.

Solving systems of PDEs, including multiscale ones, on complex
geometries and with singular data often requires the use of
unstructured grids or meshes.  Owing to this, finite element
approaches, known for their flexible use in complex geometries and
relative ease in adaptive meshing, have received substantial attention
for multiscale PDEs since the turn of the century.  Computational
power has increased exponentially, but demands on model features and
accuracy have managed to keep pace and use such power.  Currently no
single stratagem dominates the state-of-the-art.

It strongly depends in the problem at hand whether a given technique
will prove effective or not. Rather than a single technique,
multiscale modeling is a paradigm on how to relate said scales and
make them interact.  Multiscale models for PDEs can be classified in
several different ways. One classification discriminates on whether
the scales are physically embedded in each other or fully
scale-separated~\cite{weinan11,pavliotis08}.  The first class thus
formed allows for a size relation between a microscopic scale and a
macroscopic scale and they must share spatial dimension; this class of
models is studied, among others, by \cite{hellman2015multiscale} and
\cite{efendiev13}. The second class, referred to as \emph{fully
separate scale} models, allows for different space dimensions at
various scales, e.g., a one-dimensional microscopic domain within in a
three-dimensional macroscopic domain.  Our approach deals with this
second class. Other examples of this technique can be found in for
instance \cite{radu10}.

Multiscale models can also be classified in terms of how they encode
patterns in microscopic structures. Techniques derived from
homogenization often assume periodic (if not homogeneous) microscopic
structures, as for instance demonstrated in \AM{\cite{Arbogast1990}
  when deriving the double-porosity structure for a flow problem; see
  also the Ref. \cite{raats1981} for an early study in this
  direction. Among the other contributions that focus specifically on
  varying or evolving microstructures in a double-porosity case (here
  referred to as two-scale) is \cite{noorden11}.} This structure
\AM{allows} more tools for a rigorous mathematical analysis, but uses
assumptions that are often too strong for practical applications.
Since there is often an inherent variation present in microscopic
structures, investigations such as \cite{sebamPhD} focus on developing
tools that allow for the modeling of that variation\AM{, coping also
  with situations not necessarily obtainable via averaging procedures
  like homogenization. For results concerning the numerical
  approximation of such class of problems, ee refer the reader, for
  instance, to
  \cite{hou97,Varvara,hoang05,abdulle12,garttner20,Bastidas}.}

Since a higher heterogeneity in the medium inevitably leads to more
computational work, several contributions focus on improving
computational strategies for multiscale simulations.  Two of the most
popular (and often-combined) approaches are employing adaptive grids
for the discretization of the system of equations (adaptivity) and
distributing the workload over different processors (parallelization):
see for instance \cite{farhat1991method, komatitsch2010high} and
\cite{Verfurth:13:book:A-posteriori}. In our framework, we apply a
parallelization technique.  For completeness, we mention that
alternative solution approaches for multiscale flow and transport are
studied as well, see \cite[e.g.]{PesShow,raats1981,redeker13,maes20}.

The rest of the paper is structured as follows: in \Cref{sec:model} we
discuss the model PDE system and its analytic properties. In
\Cref{sec:FEM} we propose a finite element method and discuss its
implementation. In \Cref{sec:manufactured} we report on numerical
experiments conducted with this method and showcase the computational
advantages with benchmarking and comparisons, including a parallel
implementation and an application to test the dependence on the
microscopic geometry. We close the paper with some conclusions.

\section{A two-scale model with varying microstructures}
\label{sec:model}
Before posing the model problem, we introduce the two-scale setting
and the heterogeneous microscopic structure. Let $\Omega \subset
\R^{d_1}$ be the macroscopic domain and $Y \subset\R^{d_2}$ be a
microscopic domain with $d_1,d_2 \in \{1,2,3\}$. \AM{Both domains
  $\Omega$ and $Y$ are assumed to have Lipschitz boundaries.}  For
each $x \in \Omega$ we define a microscopic domain $Y_x \subset Y$
\AM{with $Y_x \cap Y=\emptyset$.}  We denote the boundary of $\Omega$
with $\partial \Omega$, consisting of mutually disjoint parts
$\partial\Omega^\Neu$ and $\partial\Omega^\Dir$ such that $\partial
\Omega = \partial\Omega^\Neu \cup \partial\Omega^\Dir $, and the
boundary of $Y_x$ with $\Gamma_x$, consisting of mutually disjoint
parts $\Gammax I$ , $\Gammax O$ and $\Gammax N$, such that $\Gamma_x =
\Gammax I \cup \Gammax O \cup \Gammax N$. The different boundaries on
the microscopic domain are presented in~\Cref{fig:schema}.  We proceed
to define the composite (multiscale) domain $\Lambda$ as
\begin{equation}
  \Lambda := \bigcup_{x\in\Omega} \{ x \} \times Y_x.
\end{equation}
Let $\zeta \in\cont0(\bar{\Omega};\cont{1}(\bar{Z},\bar{Y}))$ be a
given mapping. \OLak{Then, for each $x\in\Omega$ the microscopic
  domains $Y_x$ is the image of a base domain $Z \subset Y$, common to
  all $x$, under a mapping $\zeta(x,\cdot)$}:
\begin{equation}
  \zeta(x,Z) = Y_x.
\end{equation}
\OLak{We assume that $\zeta(x,\cdot)$ is differentiable and invertible
  in its second argument, for each $x\in\Omega$,} that
\begin{equation}
  \AM{0<}c_* \leq \mathrm{det}\nabla_y \zeta(x,\cdot) \leq c^*,
\end{equation}
It follows the partial inverse is smooth
\begin{equation}
  \zeta_y^{-1} \cont0(\bar{\Omega};\cont{1}(\bar{Y},\bar{Z})).
\end{equation}

Similar setups of heterogeneous microscopic structures are presented
in \cite{sebamPhD,lakkis13,noorden11}.  In continuum mechanics, this
concept is often referred to as motion mapping
\cite{bonet1997nonlinear}.  See \Cref{fig:schema} for a schematic
representation of $\Omega$ and $Y_x$. The boundary portions
$\Gamma_I$, $\Gamma_O$ and $\Gamma_N$ represent the inflow, outflow
and no-flow boundary of $Y_x$, respectively.

\begin{figure}%
  \begin{tikzpicture}[scale=0.4]
    \def\l{10}; \def\r{1}; \draw[black,fill=gray] plot [smooth cycle]
    coordinates {(\r*8, \r*2) (\r*9, \r*4) (\r*8, \r*7.5) (\r*5.5,
      \r*9) (\r*4, \r*8.5) (\r*2, \r*5) (\r*2.5, \r*2) (\r*7, \r*1)};
    \node at (0.6*\l,0.45*\l + 2) (macro) {\large $\Omega$}; \node at
    (0.6*\l,0.30*\l+7.5) (macroboundary) {\large $\partial \Omega$};

    \coordinate (x) at (\l*0.6, 0.3*\l); \draw[black] (2.1*\l,0.3*\l)
    coordinate (SW) -- (2.9*\l,0.2*\l) coordinate (SE)
    node[pos=0.5,below]{$\Gammax N$} -- (2.9*\l,0.8*\l) coordinate
    (NE) node[pos=0.5,right]{$\Gammax O$} -- (2.1*\l,0.7*\l)
    coordinate (NW) node[pos=0.5,above]{$\Gammax N$} -- cycle
    node[pos=0.5,left]{$\Gammax I$}; ; \path (x) node[left] {\large
      $x$}; \draw[thick,black!25] (x) -- (SW);\draw[thick,black!12.5]
    (x) -- (NE); \draw[thick,black!12.5] (x) -- (SE);
    \draw[thick,black!25] (x) -- (NW);
    \draw[black,fill=black!50,opacity=.5]
    (SW)--(SE)--(NE)--(NW)--cycle; \node at (2.5*\l,0.5*\l) (micro)
         {\large $Y_x$};
  \end{tikzpicture}
  \centering
  \caption{Schematic representation of the multiscale domain: at each
    macroscopic point $x\in\Omega$ corresponds a microscopic domain
    $Y_x$ with mixed boundary conditions (pure Neumann or
    Robin).}\label{fig:schema}
\end{figure}
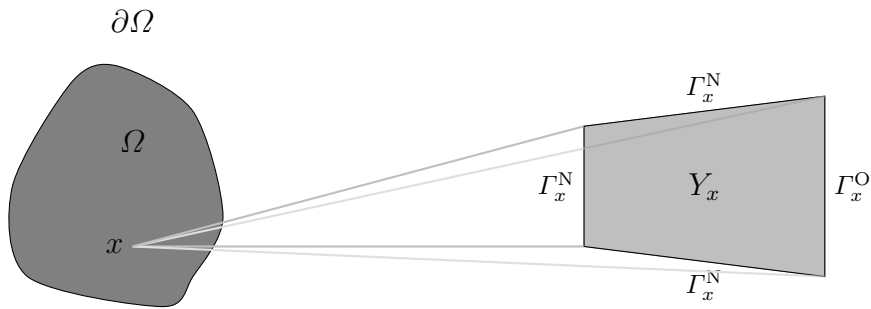

\subsection{Model problem}
The following system of equations should primarily be regarded as a
test problem for the two-scale finite element framework that is the
main focus of this work. \AM{The reader should not seek {\em per se} a
  derivation of this model via homogenization as it does not have a
  natural one. This is simply a collection of mass balance laws posed
  on different microscopic and macroscopic sets.}  In order to
illustrate the applicability of \AM{our computational} framework, we
describe an interpretation of the model as a description of transport
processes in plants mediated by cell-scale influx and efflux. Under
such an interpretation the macroscopic quantity $u$ represents a
nutrient that the plant absorbs or produces, e.g., water absorbed from
the soil. This nutrient is in turn taken in at cell membranes at an
influx part of the membrane $\Gammax I$, within the cells the nutrient
concentration is represented by microscopic quantity $v$. In the
cells, the nutrient $v$ is converted to a product $w$ which is emitted
into the tissue at outflux regions of the cell membranes
$\Gammax{O}$. Such a description is consistent with simplifications of
many models for Auxin and water transport in plants present in the
literature
\cite[e.g.]{twycross2010stochastic,chavarria2010homogenization}. As
mentioned in the introduction, a major novelty of the model and the
computational framework is that we allow for quite general cell
geometries. In~\Cref{sec:experiments} we illustrate the dependence of
both macroscopic and microscopic solution profiles on the cell
geometries.

Let $u,v:\Omega\to\R$ and $v:\Lambda\to\R$ satisfy the following
system of equations:
\begin{equation}
  \label{eq:macro-2xmicro-system}
  \begin{aligned}
    - \Delta_x u = f^u - \int_{\Gammax I} \kappa_1 u - \kappa_2 v
    \d\sigma_y &\mbox{ in } \Omega,\\ u = u_0 &\mbox{ on }
    \partial\Omega^\Dir ,\\ \nabla_x u\cdot n_\Omega = 0 &\mbox{ on }
    \partial\Omega^\Neu ,\\[2.5ex] -D^v \Delta_y v = f^v &\mbox{ in }
    \Lambda,\\ D^v\nabla_y v \cdot n_{Y_x} &= \begin{cases} \kappa_1u
      - \kappa_2v &\mbox{ on } \Gammax I\\ \kappa_3w - \kappa_4v
      &\mbox{ on } \Gammax O\\ 0 &\mbox{ on } \Gammax N
    \end{cases},\\[2.5ex]
    -\div\left(D^w \nabla_x w\right) = f^w -\int_{\Gammax 0} \kappa_3w
    - \kappa_4v \d\sigma_y &\mbox{ in } \Omega,\\ D^w \nabla_x w \cdot
    n_\Omega = 0 &\mbox { on } \partial\Omega,
  \end{aligned}
\end{equation}
where $D^w \in \leb \infty(\Omega)$, $u_0 \in \leb 2(\partial
\Omega^\Dir )$ and $D^v, \kappa_1, \kappa_2, \kappa_3, \kappa_4 \in
\R$.  In the rest of this manuscript, this system is referred to as
\sys{}.  The function space $Q$ is defined as
\begin{equation}
  Q := \sobh[D]1 (\Omega) \times \leb 2 \left( \Omega; \sobh1 (Y_x)
  \right) \times \sobh1 (\Omega),
\end{equation}
where $\sobh[D]1 (\Omega)$ is the Sobolev space with Dirichlet values
\begin{equation}
  \sobh[D]1 (\Omega) = \left\{ \phi \in \sobh1 (\Omega) \middle|
  \phi|_{\partial \Omega^\Dir} = 0 \right\}
\end{equation}
and the slightly inconsistently denoted $\leb2(\Omega,\sobh1(Y_x))$ is
rigorously defined as
\begin{equation}
  \left\{ \varphi\in\leb2(\Lambda) :
  \exists\,\phi\in\leb2(\Omega;\sobh1(Z)) :
  \varphi(x,\cdot)=\phi(x,\zeta^{-1}(x,\cdot)) \text{ a.e. }x\in\Omega
  \right\}.
\end{equation}
We construct the weak form of \sys{} by multiplying the equations with
test functions from the triplet $(\phi, \psi, \tilde{\phi}) \in Q$ and
integrating over the respective domains. This results in the following
problem: find a triplet of functions $(u,v,w) \in Q$ such that
\begin{equation}
  \begin{split}
    \int_\Omega \nabla_x u \cdot \nabla_x \phi \dx&= \int_\Omega \big(
    f^u - \int_{\Gammax I} \kappa_1 u - \kappa_2 v \d\sigma_y
    \big)\phi \dx , \\ \int_\Omega \int_{Y_x} D^v \nabla_y v\cdot
    \nabla_y \psi \dy\dx &= \int_{\Omega} \int_{Y_x} f^v \psi \dy\dx +
    \int_{\Omega} \int_{\Gammax I} (\kappa_1 u - \kappa_2 v ) \psi
    \d\sigma_y\dx\\ &+ \int_{\Omega} \int_{\Gammax O} (\kappa_3w -
    \kappa_4v) \psi \d\sigma_y \dx,\\ \int_\Omega D^w \nabla_x w \cdot
    \nabla_x \tilde{\phi} &= \int_\Omega \left( f^w - \int_{\Gammax O}
    \kappa_3w - \kappa_4v \d\sigma_y \right) \tilde{\phi} \dx,
  \end{split}
  \label{eq:system_weak}
\end{equation}
for all triplets of functions $(\phi, \psi, \tilde{\phi}) \in Q$.
\begin{definition}\label{def_solution}
  \addcontentsline{toc}{subsection}{\thesubsection.\ Weak solution definition}
  We call the triplet $(u,v,w) \in Q$ a \emph{weak solution} to \sys{} if it
  satisfies the identities listed in \eqref{eq:system_weak} for any
  choice of test functions $(\phi, \psi, \tilde{\phi}) \in Q$.
\end{definition}

\subsection{Assumptions on data}
\label{sse:assumptions}
Before proceeding with solvability of \sys{} we state the assumptions
on which we base the proof:
\begin{enumerate}
\item\label{as:lips_omega}$\Gammax I \neq \emptyset$ and $\Gammax O
  \neq \emptyset$ for all $x\in\Omega$; {$\Gammax I$, $\Gammax 0$,
    and $\Gammax N$ are Lipschitz, while $Y_x$ is convex;} %
  think Lipschitz, convex for $Y_x$ is appropriate?
\item\label{as:lips_y} $\partial\Omega^\Dir \neq \emptyset$;
\item\label{as:dirichlet} $u_0 \in \leb 2(\partial \Omega^\Dir)$;
\item\label{as:parameters} $D^w \in \leb \infty(\Omega)$, $D^w, D^v >
  0$, $f^u, f^w \in \leb 2(\Omega)$ and $f^v \in \leb 2(\Omega;\leb
  2(Y_x))$;
\item \label{as:kappa}$\kappa_1,\dots, \kappa_4 > 0$;
\item \label{as:coerc} We impose the following structural relations on
  the model parameters:
  \begin{equation}
    \begin{split}
      \frac{|\kappa_1 - \kappa_2|}{2}\left|\Gammax I\right| &<
      1.\\ \frac{|\kappa_3 - \kappa_4|}{2} \left|\Gammax O\right| &<
      \min_{x\in\bar\Omega}\{D^w(x)\}.
    \end{split}
  \end{equation}
\end{enumerate}
Assumptions~\ref{as:lips_omega}, \ref{as:lips_y} and
\ref{as:dirichlet} are geometric: they allow us to make sense of the
function spaces with distributed microstructures without too many
technicalities. Assumptions \ref{as:parameters} and \ref{as:kappa}
have a clear physical translation. They point out a set of parameters
for which solutions will turn to exist and to be
unique. Assumption~\ref{as:coerc} is a sufficient condition to prove
the coercivity of the bilinear form associated to \sys{}, without
being disrupted by the Robin boundary conditions.  It reveals the
physical interactions between the transmission coefficients
$\kappa_1,...,\kappa_4$, the size of the microstructure, and the
macroscopic diffusion coefficient.
\begin{theorem}[weak solvability]
  \addcontentsline{toc}{subsection}{\thesubsection.\ Weak solvability theorem}
  Under assumptions
  \ref{sse:assumptions}.\ref{as:lips_omega}--\ref{as:coerc} system
  \sys{} admits a unique solution in the sense of~\Cref{def_solution}
  which is also stable with respect to parameters.
\end{theorem}
\begin{proof}
  Problem \sys{} is a linear and coupled system of elliptic
  equations. A standard application of the Lax--Milgram Lemma
  (cf. Theorem 1, on p. 317 in \cite{evans10}) clarifies the
  solvability. A potentially less straightforward aspect is the
  structure of the function space $\leb2(\Omega;\sobh1(Y_x))$. Based
  on (A1) and (A2), and relying on \cite{BR92},
  $\leb2(\Omega;\sobh1(Y_x))$ is a direct integral of Hilbert spaces,
  which is itself a Hilbert space together with its corresponding
  trace spaces $\leb 2(\Omega; \leb 2(\Gammax I))$ and $\leb 2(\Omega;
  \leb 2(\Gammax O))$. For more details on this topic, we refer the
  reader to \cite{meier08} or to Section 2.2 in \cite{sebamPhD}, which
  treats the topic of Sobolev spaces in non-cylindrical domains as
  used in this framework, as well as to the more recent account
  \cite{Evseev} for the $\leb p$ version of these spaces.
\end{proof}

\section{Finite element implementation}
\label{sec:FEM}
\subsection{Approximation}
Rather than constructing a mesh for each of the microscopic domains
$Y_x$, we use the structure in $\zeta(x, \cdot)$ to construct a mesh
for the reference domain $Z$ only and pull-back the basis functions of
$Y_x$ to $Z$.  This leads us to mesh partition $\B_H$ for $\Omega$ and
mesh partition $\K_h$ for $Z$.  From these meshes we construct the
following macroscopic and microscopic finite element spaces:
\begin{equation}
  \begin{aligned}
    U_H &:= \left\{ \left. u \in
    \operatorname{C}(\bar{\Omega})\right|\,u|_B \in \mathbb{P}^1(B)
    \mbox{ for all } B \in \B_H,\, u=0 \mbox{ on } \partial
    \Omega^\Dir \right\}, \\ V_h &:= \left\{ \left. \hat{v} \in
    \operatorname{C}(\bar{Z})\right|\,\hat{v}|_K \in \mathbb{P}^1(K)
    \mbox{ for all } K \in \K_h \right\}, \\ W_H &:= \left\{ \left. w
    \in \operatorname{C}(\bar{\Omega})\right|\,w|_B \in
    \mathbb{P}^1(B) \mbox{ for all } B \in \B_H\ \right\}.
  \end{aligned}
\end{equation}
The multispace structure of the finite element space is inspired by
constructions from
\cite{lind20,radu10,MunteanLakkis:10:inproceedings:Rate}.

Let $\N_1$ denote the set of degrees of freedom for $\B_H$. Let
$\xi_i$ for $i\in \N_1$ denote a set of basis functions such that $W_H
= \operatorname{span}(\xi_i)$ and $U_H \subset
\operatorname{span}(\xi_i)$.  Furthermore, let $\N_2$ denote the set
of degrees of freedom for $\K_h$ and let $\hat{\eta_j}$ for $j\in\N_2$
denote a set of basis functions such that $V_h =
\operatorname{span}(\hat{\eta_j})$.  Note that we use hats on
functions to indicate their correspondence to the reference domain
$Z$.  Now, approximating $(u,v,w)$ in the aforementioned finite
element spaces yields:
\begin{equation}
  u(x) = \sum_{i}\vec{u}_i \xi_i(x), \quad v(x,y)= \sum_{i,j}
  \vec{v}_{ij} \xi_{i}(x)(\zeta_i\comp\eta_j)(y), \quad w =
  \sum_{i}\vec{w}_i \xi_i,
\end{equation}

In addition, writing $\D_{\hat y}\zeta(x,\hat y)$ for the Jacobian
matrix of $\zeta$ in its second argument, define
\begin{equation}
  K(x,\hat y):=[\D_{\hat y} \zeta(x,\hat y)]^{-1},
  \label{eq:grad_map}
\end{equation}
and
\begin{equation}
  J(x,\hat y):= \det{\D_{\hat y} \zeta(x,\hat y)},
  \label{eq:det_map}
\end{equation}
to be able to express functionals of a generic function $v$ with
domain $Y_x$ as a functionals of
\begin{equation}
  \hat{v}(x,\hat y)=v(x,\zeta(x,\hat y))\text{ on $Z$}.
\end{equation}
Specifically, we use of the following change of variables:
\begin{equation}
  \begin{aligned}
    \int_{Y_x} v\d y =& \int_Z \hat vJ\d\hat{y}, \\ \int_{Y_x}
    \nabla_y v\d y =& \int_{Z} K^T \nabla_{\hat{y}} \hat{v}J
    \d\hat{y}, \\ \AM{ \int_0^T\int_{\partial Y_x} q v \d\sigma_y dt}
    =& \AM{\int_0^T\int_{\partial Z} |{\rm cof}(F)\nu_{\hat y}| \hat q
      \hat{v} \d\sigma_{\hat{y}} dt,}
  \end{aligned}
\end{equation}
\AM{which hold almost everywhere on $(0,T)\times\Omega$ for any
  function $\hat v\in\leb 2((0,T)\times\Omega; \sobh1 (Z))$ and $\hat
  q\in \leb 2((0,T)\times\Omega; \leb 2 (Z))$. Here ${\rm cof}(F):=
  ({\rm det} F) F^{-T}$ is the cofactor matrix of $F$, where $F$
  denotes here the Jacobian of the wanted transformation. }

Discretizing \eqref{eq:system_weak} and applying the transformations
above we obtain a discrete weak formulation where the microscopic
contributions only require reference cell $Z$:
\begin{equation}
  \begin{split}
    &\sum_{i}\int_\Omega \vec{u}_i\nabla_x \xi_i \nabla_x \xi_k\dx+
    \int_\Omega \vec{u}_i \xi_i \xi_k \int_{\Gamma^I}\kappa_1
    J\d\sigma_{\hat{y}}\dx \\ &\quad = \int_\Omega f^u - \left(
    \int_{\Gamma^I} (- \kappa_2 \vec{v}J \d\sigma_y \right) \xi_k \dx
    ,\\ &\sum_{i,j} \int_{\Omega \times Z} \vec{v}_{ij} D^v
    \xi_{i}\nabla_y \hat{\eta_j} KK^T\nabla_y \hat{\eta_l}
    \xi_{k}Jd\hat{y}\dx + \int_{\Omega \times \Gamma^I} \kappa_2
    \vec{v}_{ij} \xi_{i}\xi_{k} \hat{\eta_j}
    \hat{\eta_l}J\d\sigma_{\hat{y}}\dx \\ &\quad+ \int_{\Omega \times
      \Gamma^O} \kappa_4 \vec{v}_{ij} \xi_{i}
    \xi_{k}\hat{\eta_j}\hat{\eta_l}J\d\sigma_{\hat{\eta_l}}\dx \\ &
    \quad=\int_{\Omega \times Z} f^v \xi_{k}\hat{\eta_l}Jd\hat{y}\dx +
    \int_{\Omega \times \Gamma^I}
    \kappa_1\vec{u}\xi_{k}\hat{\eta_l}J\d\sigma_{\hat{y}}\dx +
    \int_{\Omega \times \Gamma^O} \kappa_3\vec{w}
    \xi_{k}\hat{\eta_l}J\d\sigma_{\hat{y}}\dx ,\\ &\sum_{i}\int_\Omega
    \vec{w}_iD^w\nabla_x \xi_i \nabla_x \xi_k + \int_\Omega \vec{w}_i
    \xi_i \xi_k \int_{\Gamma^O}\kappa_3J\d\sigma_{\hat{y}} \\ &\quad=
    \int_\Omega f^w - \left(\int_{\Gamma^O}- \kappa_4\vec{v}J
    \d\sigma_{\hat{y}}\right) \xi_k \dx .
  \end{split}\label{eq:discrete_weak_micro_reference}
\end{equation}
Using techniques similar to the ones presented in for instance
\cite{lind20}, it is possible to prove \emph{a priori} convergence
rates of the finite element approximation. We do not provide such an
analysis in this manuscript.
\subsection{Implementation}
We implement \eqref{eq:discrete_weak_micro_reference} using the finite
element library \dealii{} \cite{dealii}, a C++ library that
facilitates the creation and management of finite element
implementations. The project is released under an open-source license
and under active development. The choice for this library is motivated
by its features, which include the possibility of
dimension-independent implementations, support for many different
elements and support for different parallelization options.  Since
\dealii{} does not directly support multiscale systems of PDE, we
extend it with support for multiscale function and tensor objects.
The implementation is available on GitHub
\cite{Richardson:21:url:Finite}.
\subsection{Structure of FEM implementation}
A classic finite element implementation in \dealii{} is roughly
structured along the following steps:

\begin{enumerate}
\item Load a domain and generate a triangulation of that domain.
\item Distribute the degrees of freedom based on the triangulation and
  the degree of finite elements.
\item Initialize the system matrix and the right hand side and
  solution vector.
\item Assemble the linear system by looping over cells and integrate
  the discrete weak form for each quadrature point.
\item Apply the boundary conditions of the PDE to the linear system.
\item Precondition and solve the system, either directly or
  iteratively.
\item Post-process and output the solution.
\end{enumerate}

In our finite element framework, we define a microscopic system for
each macroscopic degree of freedom. More precisely, since we use basis
functions with local support, the support point of the macroscopic
basis function represents the location of the microscopic
system. Correspondingly, we associate a specific local value of $u$
with each system: the finite element function value in the coordinates
of the support point.  This allows us to apply the same finite element
framework for the microscopic functions over the macroscopic domains,
creating in essence a tensor product of finite element spaces.  The
system is solved by decoupling the microscopic and macroscopic
formulation. Iteratively, we solve first the macroscopic system, using
the microscopic solutions of the previous iterations as data, followed
by solving the microscopic system, using the most recently obtained
macroscopic as data. We do so until we observe the difference between
subsequent residuals is within a predetermined tolerance.

From a computational point of view, the effort associated with the
creation and solving of microscopic systems can quickly become
large. Applications that require solutions with high accuracies place
demands on both memory and CPU capacity. For this reason, we address
the computational load by customizing the implementation for
multiscale systems using caching techniques, domain mappings,
multithreaded assembly and distributed solving.  Furthermore, we
exploit the similarity between the different microscopic scales by
using the same data structures for the microscopic systems, where
possible.

\subsection{Domain mappings}
\label{sec:mapping}
Mapping $\zeta$ is, as described in \Cref{sec:model} implemented as a
bounded, not necessarily linear function of both $x$ and $y$.  $\zeta$
is supplied symbolically to the implementation, as well as the inverse
of the Jacobian $K(x,y)$ (see \eqref{eq:grad_map}) We avoid
recomputing these quantities by computing the structures $J$ and
$KK^TJ$ for each quadrature point and storing them in a hashmap for
fast access throughout the iterative assembly procedure.
Additionally, we apply the mapping in the multiscale functions, so
that by using the same assembly objects for each $x$, we can evaluate
all microscopic data on $Y_x$.

\subsection{Parallelization in assembly}
As mentioned before, we parallelize the implementation using
multithreading.  We divide computational the work into different
threads which are assigned to processors that share memory, meaning
that there is no need to specify communication between processors but
that special care is needed to avoid so-called 'race conditions',
where errors are made due to multiple threads writing to the same
memory structure at the same time, or one memory structure reading
from the memory while another one is updating it.

A commonly used alternative is a distributed memory approach (based on
for instance a message passing interface like MPI) where memory is
distributed.  This has the advantage of being able to be run on
clusters of computing nodes that do not share a memory space, but has
the disadvantage that data needs to be communicated explicitly and
that often leads to communication overhead and idle time when a
processor on a node is unused.  The advantages of using a
multithreading approach is that our settings lends itself well for
doing so; there is minimal setup and overhead involved in the
parallelization and the multiscale nature of the implementation
reveals a structure that lends itself well for threads.  A
disadvantage is that most larger high performance clusters are
composed of smaller compute nodes, requiring some sort of message
passing.  However, we remark that using a combination of
multithreading and MPI is a valid strategy as well and something that
could be applied on this setup too, as we discuss in
\Cref{sec:conclusions}.

We implement a multithreading approach using the Threading Building
Blocks library, for which \dealii{} has integrations in place.  In the
assembly loop, we loop over all cells in reference domain $Z$ to
compute the contributions of the associated degrees of freedom to the
system matrix and the system right hand side vector.  Rather than
consecutively looping over all cells per microscopic system, we
compute the contributions of all cells that correspond to the cell in
the reference domain and add them to their respective system matrices.
We use a thread-pool that assigns a new cell to a processor as soon as
that processor has finished work on its previous cell, since all
microscopic systems can be assembled independently.  The result of
this assembly step is a linear system for each microscopic finite
element formulation.  Both the macroscopic and microscopic systems are
assembled in parallel, with a per-cell approach.  This process can be
run in parallel, since computing the contribution to the system matrix
and the system right hand side of one finite element cell does not
require information from other cells.  This yields a linear system for
each microscopic formulation.

The residual threshold used in this framework is the sum of the
macroscopic and average microscopic residuals.  The solution procedure
is standard: we use the conjugate gradient method to solve the
systems. The systems are also solved in parallel, by distributing them
over the available processors.
\subsection{Visualization}
One of the most important post-processing steps of a simulation is
visualizing the results. We make use of the visualization toolbox
ParaView \cite{paraview} and the associated library VTK \cite{vtk}.
The multiscale nature of the problem creates non-standard demands for
the visualizations. In particular, since we solve the microscopic
systems on the same reference domain, visualizing all of them at once
is no longer a trivial task.

We need to push-forward the solution to the correct domain $Y_x$,
since the microscopic systems are solved on $Z$. \dealii{} outputs
data in a VTK-compliant format: an unstructured grid format
(\cite{vtk}). By applying the mapping on the transformation, we obtain
a visual representation of the solution on the correct domain.  In
addition, we apply a translation to the microscopic unit domains to
move them to their macroscopic location and scale for visualization
purposes. Finally, ParaView can read these data files and render a
visualization that provides an intuitive overview of the domain.  The
solutions to \sys{} are visualized using this technique and displayed
in \Cref{fig:u-case-a,fig:v-case-b}.
\section{Numerical experiments}
\label{sec:manufactured}
We test and benchmark the implementation by applying the method of
manufactured solutions on a system of PDE that generalizes \sys{}.
This allows us to see if our numerical scheme converges with the
desired order, by prescribing a solution and deriving the
corresponding PDE data.  The manufactured system is formulated as
follows:
\begin{equation}
  \begin{aligned}
    - \Delta u & = - \int_{\Gammax I} \kappa_1 u - \kappa_2 v + g_1^v
    \d\sigma_y + f^u && \mbox{ in } \Omega, \\ u & = u_0 + g^u_1 &&
    \mbox{ on } \partial\Omega^\Dir , \\ \nabla u\cdot n_\Omega & =
    g^u_2 && \mbox{ on } \partial\Omega^\Neu , \\-D^v \Delta v & = f^v
    && \mbox{ in } \Lambda, \\ D^v\nabla_y v \cdot n_{Y_x} &=
    \begin{cases}
      \kappa_1u - \kappa_2v + g^v_1 \\ \kappa_3w - \kappa_4v + g^v_2
      \\ g^v_3
    \end{cases}
    &&
    \begin{matrix}
      \mbox{on }\Gammax I, \\ \mbox{on }\Gammax O, \\ \mbox{on
      }\Gammax N,
    \end{matrix}
    \\-\div\left(D^w \nabla w\right) & = -\int_{\Gammax O} \kappa_3w -
    \kappa_4v + g_2^v \d\sigma_y + f^w && \mbox{ in }\Omega, \\ D^w
    \nabla w \cdot n_\Omega & = g^w && \mbox{ on }\partial\Omega.
  \end{aligned}
  \label{eq:manu_pde_formulation}
\end{equation}
\subsection{SymPy}
To facilitate the testing and benchmarking of the implementation, we
build an interface that given a pair of manufactured solutions $(u,v)$
and a mapping $\zeta$ constructs the data of \manusys{}.  This
interface uses the symbolic algebra package SymPy \cite{sympy} and
outputs all required functions in a \dealii{}-compliant parameter
file. This has three advantages:
\begin{itemize}
\item It eliminates the step of manually constructing problem sets.
\item No recompilation is necessary to solve different PDE systems.
\item It simplifies reproducibility and facilitates testing for
  different types of systems.
\end{itemize}
Because the mapping has an explicit form, we can derive the required
functions in \eqref{eq:manu_data} in the correct domains by composing
the functions with $\zeta$.

By computing a parametrization of $\Gammax I, \Gammax O$ and $\Gammax
N$, we can also construct the contributions of the integral terms in
the right hand sides of \manusys{}.  The interface is available as
part of the supplementary material for the manuscript.
\subsection{Convergence benchmarking}
We benchmark how the solution of the implementation converges by
solving \manusys{} for consecutively finer microscopic and macroscopic
grids and comparing the numerical error of the approximations.  We
manufacture the following solution:
\begin{equation}
  \begin{split}
    u(x_0, x_1) &= x_1\sin(x_0), \\ v(x_0, x_1, y_0, y_1) &= x_0 + x_1
    + y_0y_1(1-y_1),\\ w(x_0, x_1) &= x_0\cos(x_1),\\ \zeta(x_0, x_1,
    y_0, y_1) &=
    \begin{pmatrix}
      0.4((x_0 + 1.3)+(1.4y_0 - 0.54y_1)) \\ 0.3((x_1 + 1.2)+(-0.4y_0
      + 0.8y_1))
    \end{pmatrix}.
  \end{split}
  \label{eq:manu_system}
\end{equation}
The corresponding data used to solve the system can easily be derived
from \eqref{eq:manu_system} and is presented below.  The truncations
indicated with $\approx$ are only for legibility purposes.
\begin{equation}
  \begin{split}
    f^u &= D^v \left( \frac{324}{5^5} x_0 x_1 - \frac{108}{5^6} x_0 +
    \frac{3^3}{5^4}x_1^2 - \frac{252}{5^6}x_1 -
    \frac{48762}{5^7}\right) + x_1\sin(x_0),\\ f^v &= 2D^vy_0 ,\\ f^w
    &= D^wx_0\cos(x_1) - D^v \left( \frac{324}{5}x_0x_1 +
    \frac{1404}{6} x_0 - \frac{3^3}{5^4} x_1^2 - \frac{3204}{5^6} x_1
    + \frac{131202}{5^8} \right) ,\\ g^u_1 &= x_1\sin(x_0) ,\\ g^u_2
    &= x_1\sin(x_0) ,\\ g^v_1 &\approx D^v (0.6690 y_0 y_1 - 0.6690
    y_0 (1 - y_1) - 0.7432 y_1 (1 - y_1)) \\ \quad &-\kappa_1 x_1
    \sin(x_0) + \kappa_2 (x_0 + x_1 + y_0 y_1 (1 - y_1)) ,\\ g^v_2
    &\approx D^v (-0.6690 y_0 y_1 + 0.6690 y_0 (1 - y_1) + 0.7432 y_1
    (1 - y_1)) \\ \quad &-\kappa_3 x_0 \cos(x_1) + \kappa_4 (x_0 + x_1
    + y_0 y_1 (1 - y_1)) ,\\ g^v_3 &\approx D^v (-0.9778 y_0 y_1 +
    0.9778 y_0 (1 - y_1) + 0.2095 y_1 (1 - y_1)) ,\\ g^w &= D^w x_0
    \cos(x_1) .
  \end{split}
  \label{eq:manu_data}
\end{equation}

This system is solved for $x = (x_0, x_1) \in \Omega = [-1,1]^2$ and
microscopic domains generated by $\zeta$ from \eqref{eq:manu_system}
applied to $y = (y_0, y_1) \in Y_x \subset Z = [-1,1]^2$.  Let us
define the following macroscopic and microscopic error norms:
\begin{align*}
  e_{uw} &:= \|u - u_H\|_\MLL + \|w - w_H\|_\MLL,\\ e_v &:= \|v -
  v_{H,h}\|_\mLL,\\ e_{uw}^\nabla &:= \|u - u_H\|_\MH + \|w -
  w_H\|_\MH,\\ e_{v}^\nabla &:= \|v - v_{H,h} \|_\mH.
\end{align*}
The macroscopic and microscopic errors of the benchmark for
subsequently smaller grids are presented in
\Cref{tab:macro_convergence,tab:micro_convergence}, respectively.  The
error norms in the two tables are the result of the same simulation,
meaning that the order of the rows in the tables correspond to each
other. The mDoFs in \Cref{tab:micro_convergence} is the number of
microscopic degrees of freedom for a single microscopic formulation.
Note that this implies that, for instance, in case of 81 MDoFs, each
with 81 mDoFs, the collective number of degrees of freedom is $81^2 =
6561$.
\begin{table}[ht]
  \centering
  \begin{tabular}{cccccc}
    \toprule MDoFs & $H$ & $e_{uw}$ & $e_{uw}^\nabla$ & $p_M$ & $q_M$
    \\ \midrule 81 & \num{2.500e-01} & \num{7.115e-03} &
    \num{4.594e-02} & - & - \\ 144 & \num{1.818e-01} & \num{3.833e-03}
    & \num{3.257e-02} & 2.150 & 1.196 \\ 289 & \num{1.250e-01} &
    \num{1.794e-03} & \num{2.195e-02} & 2.180 & 1.133 \\ 576 &
    \num{8.696e-02} & \num{8.563e-04} & \num{1.509e-02} & 2.145 &
    1.087 \\ 1089 & \num{6.250e-02} & \num{4.492e-04} &
    \num{1.077e-02} & 2.026 & 1.059 \\ 2116 & \num{4.444e-02} &
    \num{2.225e-04} & \num{7.630e-03} & 2.115 & 1.038 \\ 4225 &
    \num{3.125e-02} & \num{1.124e-04} & \num{5.351e-03} & 1.975 &
    1.026 \\ 8464 & \num{2.198e-02} & \num{5.458e-05} &
    \num{3.758e-03} & 2.079 & 1.017 \\ 16641 & \num{1.562e-02} &
    \num{2.366e-05} & \num{7.162e-03} & 2.112 & 1.008 \\ \bottomrule
  \end{tabular}
  \caption{Macroscopic error and convergence rates for the
    manufactured problems. $p_M$ represents the subsequent observed
    order of convergence of the finite element error, $q_M$ represents
    the subsequent observed order of convergence of the error of its
    gradient.}
  \label{tab:macro_convergence}
\end{table}

\begin{table}[ht]
  \centering
  \begin{tabular}{cccccc}
    \toprule mDoFs & $h$ & $e_v$ &$e_v^\nabla$ & $p_m$ &
    $q_m$\\ \midrule 81 & \num{2.500e-01} & \num{6.191e-03} &
    \num{1.149e-01} & - & - \\ 144 & \num{1.818e-01} & \num{3.188e-03}
    & \num{8.344e-02} & 2.307 & 1.112 \\ 289 & \num{1.250e-01} &
    \num{1.531e-03} & \num{5.733e-02} & 2.106 & 1.078 \\ 576 &
    \num{8.696e-02} & \num{7.575e-04} & \num{3.987e-02} & 2.041 &
    1.053 \\ 1089 & \num{6.250e-02} & \num{3.807e-04} &
    \num{2.865e-02} & 2.160 & 1.038 \\ 2116 & \num{4.444e-02} &
    \num{1.991e-04} & \num{2.037e-02} & 1.952 & 1.027 \\ 4225 &
    \num{3.125e-02} & \num{9.487e-05} & \num{1.432e-02} & 2.144 &
    1.019 \\ 8464 & \num{2.198e-02} & \num{4.831e-05} &
    \num{1.007e-02} & 1.943 & 1.014 \\ 16641 & \num{1.562e-02} &
    \num{2.812e-05} & \num{2.670e-03} & 1.962 & 1.011 \\ \bottomrule
  \end{tabular}
  \caption{Macroscopic error and convergence rates for the
    manufactured problems. The microscopic degrees of freedom (mDoFs)
    are counted for a single microscopic system. $p_m$ represents the
    subsequent observed order of convergence of the finite element
    error, $q_m$ represents the subsequent observed order of
    convergence of the error of its gradient.}
  \label{tab:micro_convergence}
\end{table}
In Figure~\ref{fig:conv_benchmark} we present a graphical
interpretation of the error.
\begin{figure}[!ht]
  \centering \includegraphics[width=0.6\linewidth]{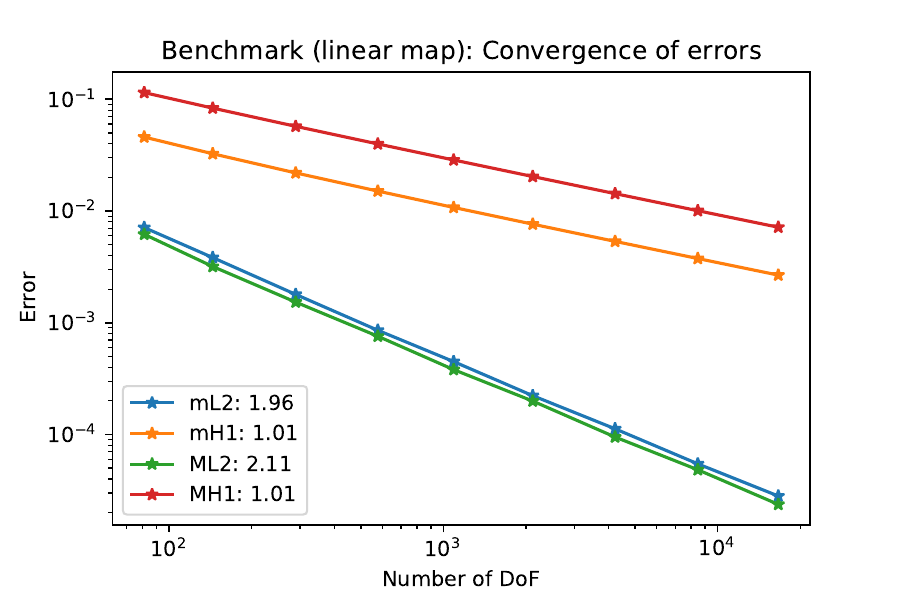}
  \caption{Approximation error as a function of the degrees of
    freedom. The observed order of convergence is computed from the
    final step}\label{fig:conv_benchmark}
\end{figure}
The convergence of the numerical scheme behaves as expected; The
behavior of $e_{uv}$ and $e_{w}$ indicates that the finite element
solution converges quadratically to the solution of the
PDE. Furthermore, $e_v^\nabla$ and $e_{uw}^\nabla$ indicate that the
gradient of the finite element solution converges linearly to the
gradient of the solution, also according to expectation.

\subsection{Parallel benchmarking}
We assess the scalability of the implementation by running it in on a
multicore CPU part of the Swedish HPC cluster Kebnekaise.  This CPU
consists of 28 Intel Xeon nodes with a base frequency of 2.60 GHz.  In
order to test the parallel scalability of the implementation for
different configurations we run two cases: A fine macroscopic grid
with coarse microscopic systems and a coarse macroscopic grid with
fine microscopic grids.  In the first case, the macroscopic system has
4225 degrees of freedom while the microscopic systems each have 81
degrees of freedom. In the second case, the macroscopic system has 81
degrees of freedom while the microscopic systems each have 4225
degrees of freedom.

\begin{figure}[!ht]
  \centering
  \includegraphics[width=0.7\linewidth]{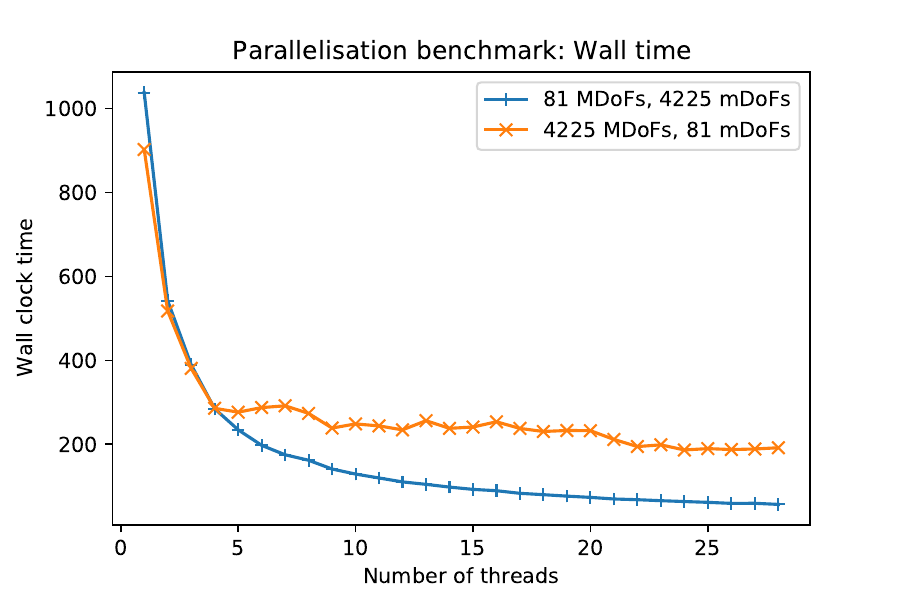}
  \caption{Wall time (total duration of the simulation) as a function
    of the number of nodes for a fine-macro/coarse-micro system and a
    coarse-macro/fine-micro system.}\label{fig:wall-time}
\end{figure}
\begin{figure}
  \centering
  \includegraphics[width=0.7\linewidth]{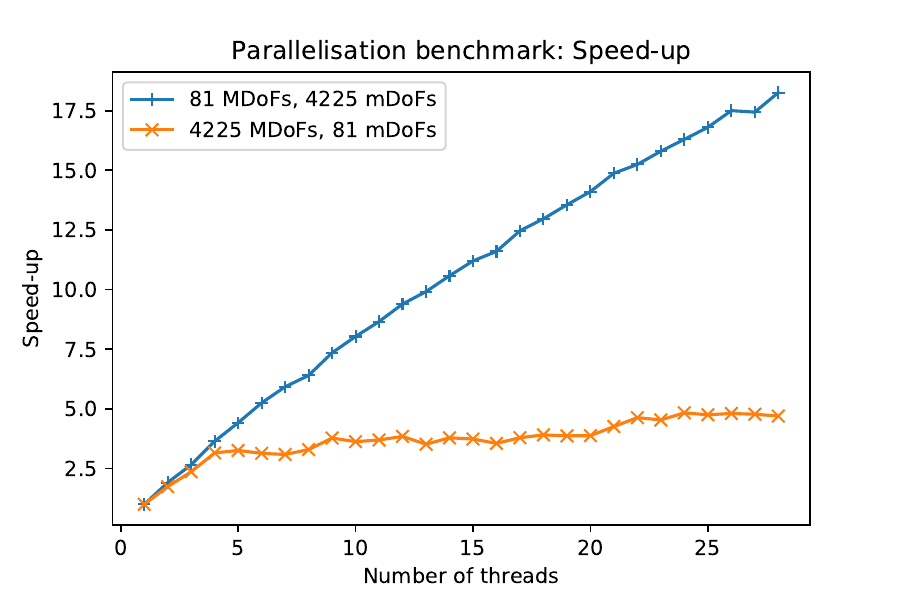}
  \caption{Speed-up as a function of the number of nodes for a
    fine-macro/coarse-macro system and a coarse-macro/fine-micro
    system.}\label{fig:speedup}
\end{figure}

Figure~\ref{fig:wall-time} displays the duration of the simulation as
a function of the number of threads for both configurations.  We
observe great parallelization performance for runs with fine
microscopic grids, but rather poor scalability for runs with fine
macroscopic grids.  This observation is supported by
Figure~\ref{fig:speedup}, displaying the speed-up for increasing
numbers of nodes.  Here it shows that the speed-up of the fine
macroscopic grid trails off after 4 threads, while the fine
microscopic grids remains scalable at the maximum tested number of 28
threads.

That the parallel performance favors a system with a few fine
microscopic grids is a consequence of the implementation. The large
microscopic grids mean that there are a lot of individual cells to be
assembled, resulting in a lot of independent tasks. This is where a
thread-pool performs best.

If one would face the opposite situation; a lot of coarse microscopic
grids, one can rely on alternative parallelization strategies. By
parallelizing per microscopic system rather than parallelizing per
cell, the large number of tasks would once again be independent. The
downside of this strategy is that it since the domain deformations are
non-linear, it requires duplication of the finite element assembly
objects for each microscopic grid, which would place a much larger
demand on memory and increases computational load as well.
\label{sec:experiments}
\subsection{Dependence of the microscopic geometry}
Having tested and benchmarked the implementation, we use it to solve
\sys{} and examine the dependence of the solution of the microscopic
geometry.

We solve \sys{} using the data presented in Table~\ref{tab:num-exp},
using the following domain deformation maps.
\begin{equation}
  \begin{split}
    \zeta_0(x,y) &:= \begin{pmatrix}
      y_1\\y_2\end{pmatrix},\\ \zeta_1(x,y) &:=
      \frac{1}{20} \begin{pmatrix}- 4x_2y_2 + 5y_1(2x_1 +
        3)\\ 5y_2(2-x_2)\end{pmatrix}.
  \end{split}
  \label{eq:data-num-exp}
\end{equation}
We investigate two settings: Case $A$ and Case $B$.  As before,
$\Omega = [-1,1]^2$, $Z = [-1,1]^2$. We model $\partial \Omega^\Dir =
-1 \times \left( -1, 1 \right)$ as a constant source of nutrient
$u$. The macroscopic finite element system has 81 degrees of freedom
and each microscopic system has 4225 degrees of freedom.  In both
cases there is no bulk production of any of the species whilst the
Dirichlet boundary condition imposed on the left-hand boundary is
taken constant. The lack of bulk production terms is enforced to
highlight the role played by the microstructure on the macroscopic
profiles. In Case $A$ the mapping describing the microstructure is
taken to be the identity whilst in Case $B$ a heterogeneous mapping is
chosen which results in more elongated cells as the distance from the
Dirichlet boundary increases with cell-size increasing from top to
bottom. The map in Case $B$ is designed to caricature the
heterogeneity observed due to distance from root tips in real plant
tissues \cite{lockhart1965analysis}. In both cases we set
$D^w\ll{D}^v$ to account for the fact that intracellular transport is
much faster than transport within the tissue.
\begin{table}[ht]
  \centering
  \begin{tabular}{ccccccc}
    \toprule & Case $A$ & Case $B$ \\ \midrule $D^w$ & 0.1 & 0.1
    \\ $D^v$ & 1 & 1 \\ $\kappa_1$ & 0.5 & 0.5 \\ $\kappa_2$ & 1 & 1
    \\ $\kappa_3$ & 0.25 & 0.25 \\ $\kappa_4$ & 1 & 1 \\ $f^u$ & 0 & 0
    \\ $f^w$ & 0 & 0 \\ $u_0$ & 1 & 1 \\ $\zeta(x,y)$ & $\zeta_0(x,y)$
    & $\zeta_1(x,y)$ \\ \bottomrule
  \end{tabular}
  \caption{Parameter sets used in numerical experiments of
    \sys. Symbol definitions are presented in
    \eqref{eq:data-num-exp}. Solutions to each case are plotted in
    Figures \ref{fig:u-case-a}--\ref{fig:v-case-b}.}
  \label{tab:num-exp}
\end{table}

The solutions to \sys{} for the cases presented in
Table~\ref{tab:num-exp} are displayed in
Figures~\ref{fig:u-case-a}--\ref{fig:v-case-b}.  The effect of the
domain deformation is evident. Both on a microscopic level (as
displayed in Figure \ref{fig:v-case-b}) and a macroscopic level (as
seen in Figure \ref{fig:w-case-b}).  We observe a quantitative as well
as a qualitative effect of the mapping in the macroscopic solutions,
specifically the heterogeneous mapping of Case $B$ destroys the
symmetry of the problem along the horizontal center line of the
(macroscopic) domain and leads to a larger gradient in the
concentration of the product of the cells $w$ (compare the scales in
Figures \ref{fig:w-case-a} and \ref{fig:w-case-b}). The main driver
here is the effect the mapping has on the volume and the aspect ratio
of the macroscopic cells. This preliminary numerical investigation
therefore suggests that plant cell geometries could act as a
significant determinant of tissue level concentration profiles in
transport processes in biological tissues.

From a modeling perspective, the coupling of the system comes in play
in two places: the right hand side of the macroscopic equations and
the boundary terms of the microscopic equations.  By tweaking the
parameters, the coupling can be made arbitrarily tight.  For instance;
taking the limit of $D^w$ to 0 yields an algebraic relationship
between $v$ and $w$.  By contrast, choosing $\kappa_i \to 0$ for
$i=1,...,4$ eliminates the relation between $u, w$ and $v$ and results
in a decoupled system.

One of the downsides of this modeling approach is that if the mapping
is anisotropic, the resulting degrees of freedom are not
equidistributed. This can cause some issues in the accuracy of the
microscopic approximations.
\begin{figure}[tp]
  \begin{minipage}[]{0.48\textwidth}
    \centering \includegraphics[width=\linewidth]{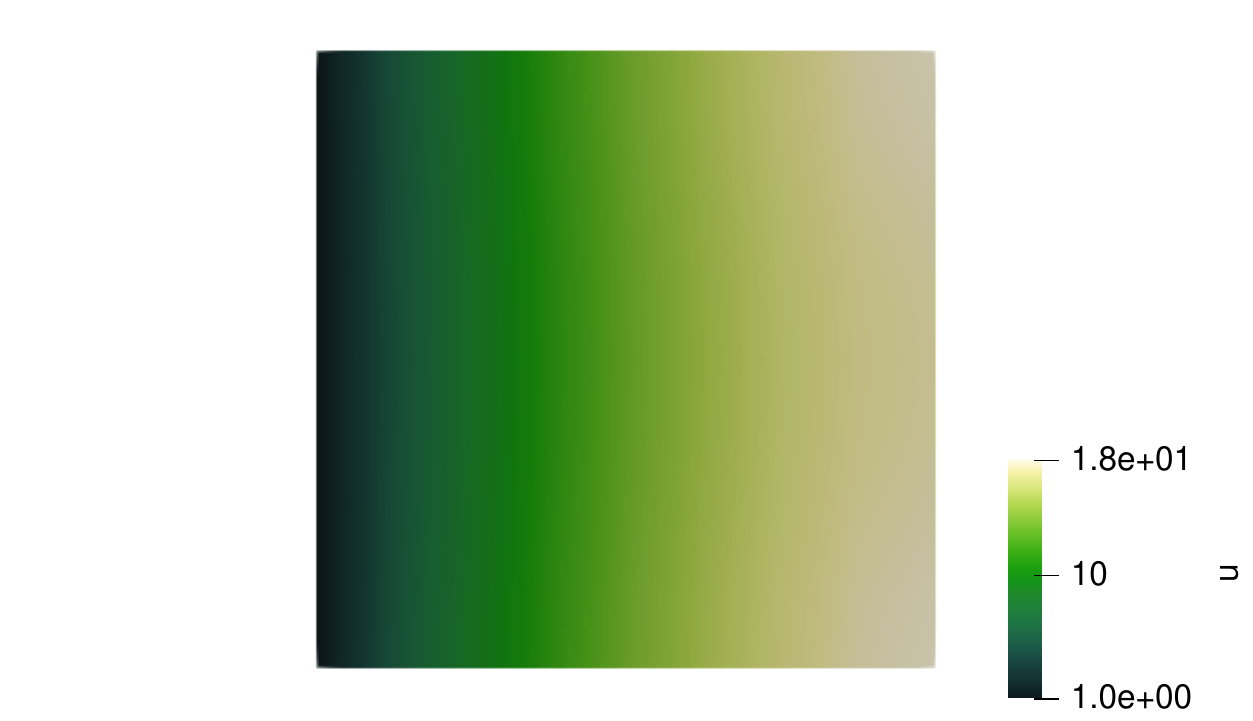}
    \caption{Nutrient concentration $u(x)$ in Case $A$, with the
      Dirichlet boundary (left) as the only source of
      nutrient.}\label{fig:u-case-a}
  \end{minipage}\hfill
  \begin{minipage}[]{0.48\textwidth}
    \centering \includegraphics[width=\linewidth]{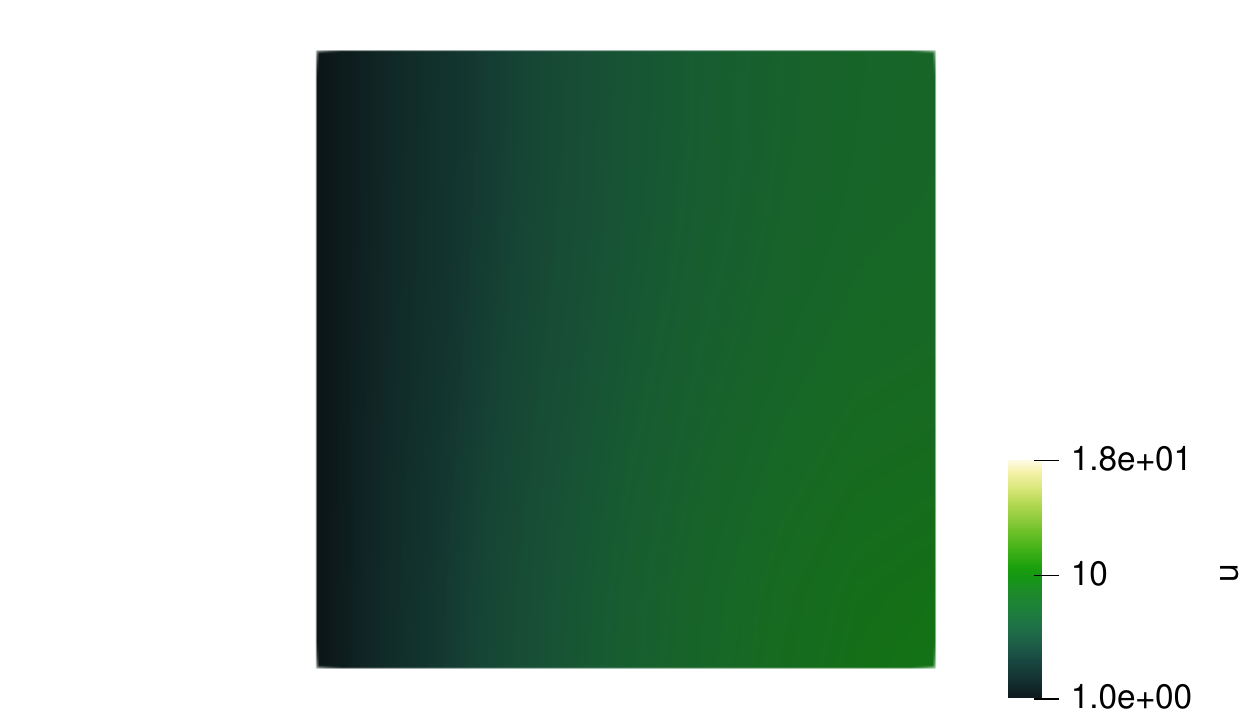}
    \caption{Nutrient concentration $u(x)$ in Case $B$, with the
      Dirichlet boundary (left) as the only source of
      nutrient.}\label{fig:u-case-b}
  \end{minipage}\hfill
  \begin{minipage}[]{0.48\textwidth}
    \centering \includegraphics[width=\linewidth]{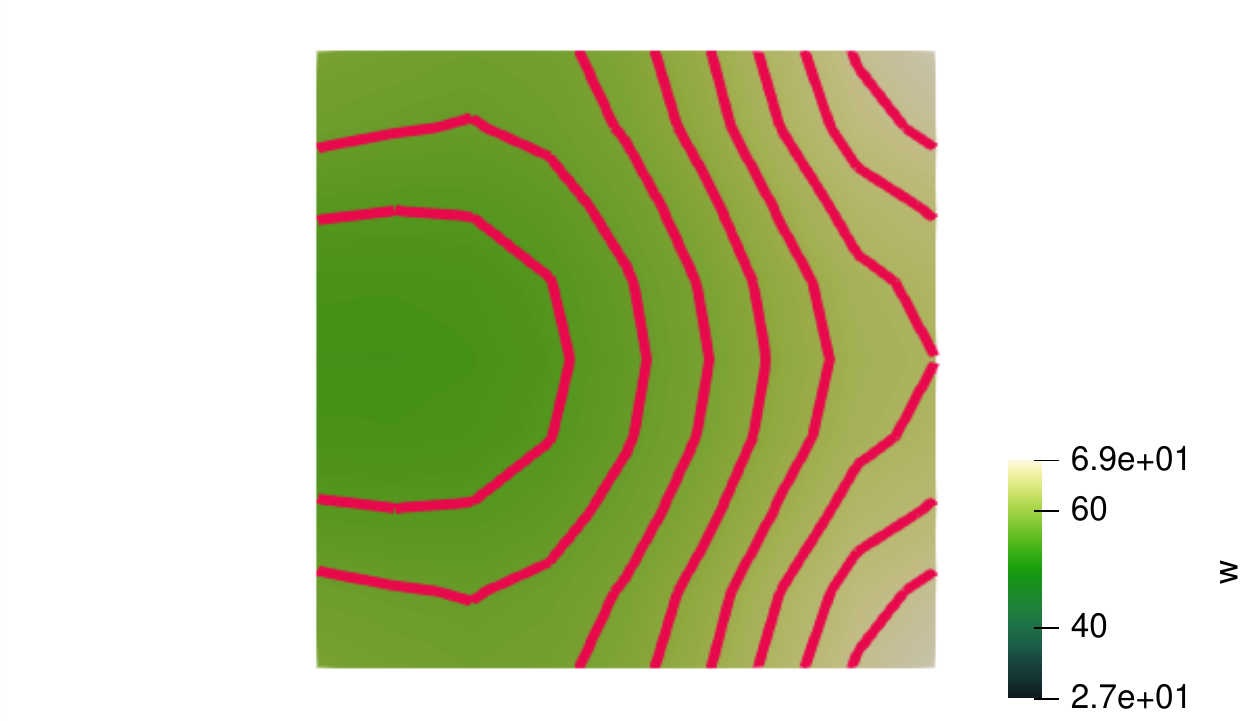}
    \caption{Concentration of the cell product $w(x)$, with the only
      source of the product the cell-based transmission from
      $v(x,y)$. Contour indicates level sets.}\label{fig:w-case-a}
  \end{minipage}\hfill
  \begin{minipage}[]{0.48\textwidth}
    \centering \includegraphics[width=\linewidth]{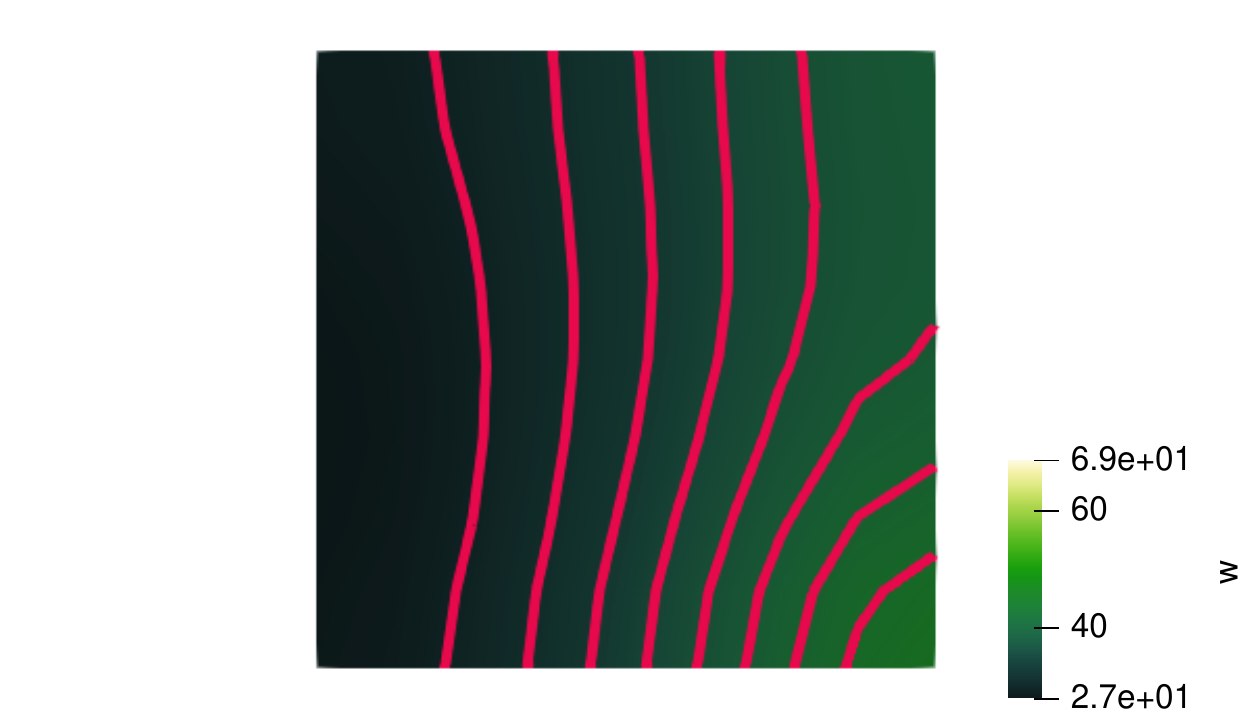}
    \caption{Concentration of the cell product $w(x)$, with the only
      source of the product the cell-based transmission from
      $v(x,y)$. Contour indicates level sets.}\label{fig:w-case-b}
  \end{minipage}\hfill
  \begin{minipage}[]{0.48\textwidth}
    \centering \includegraphics[width=\linewidth]{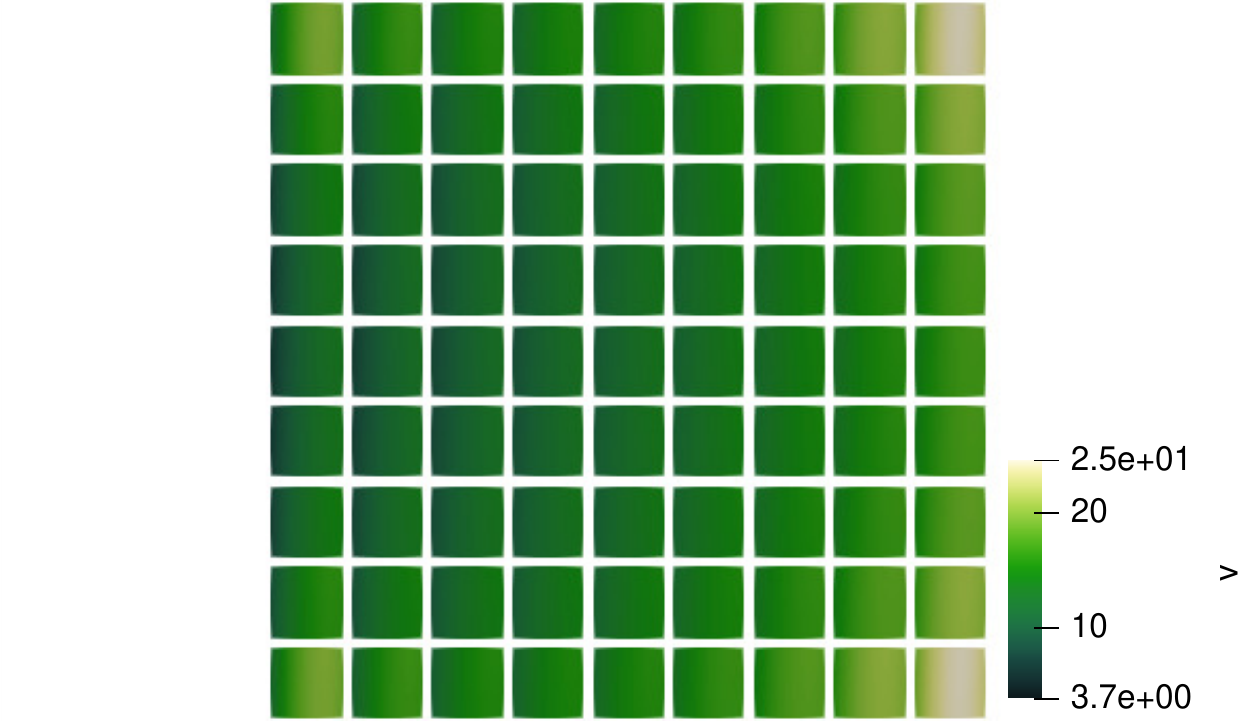}
    \caption{Cell-based nutrient concentration $v(x,y)$ in Case $A$
      with an identity domain mapping}\label{fig:v-case-a}
  \end{minipage}\hfill \begin{minipage}[]{0.48\textwidth}
    \centering \includegraphics[width=\linewidth]{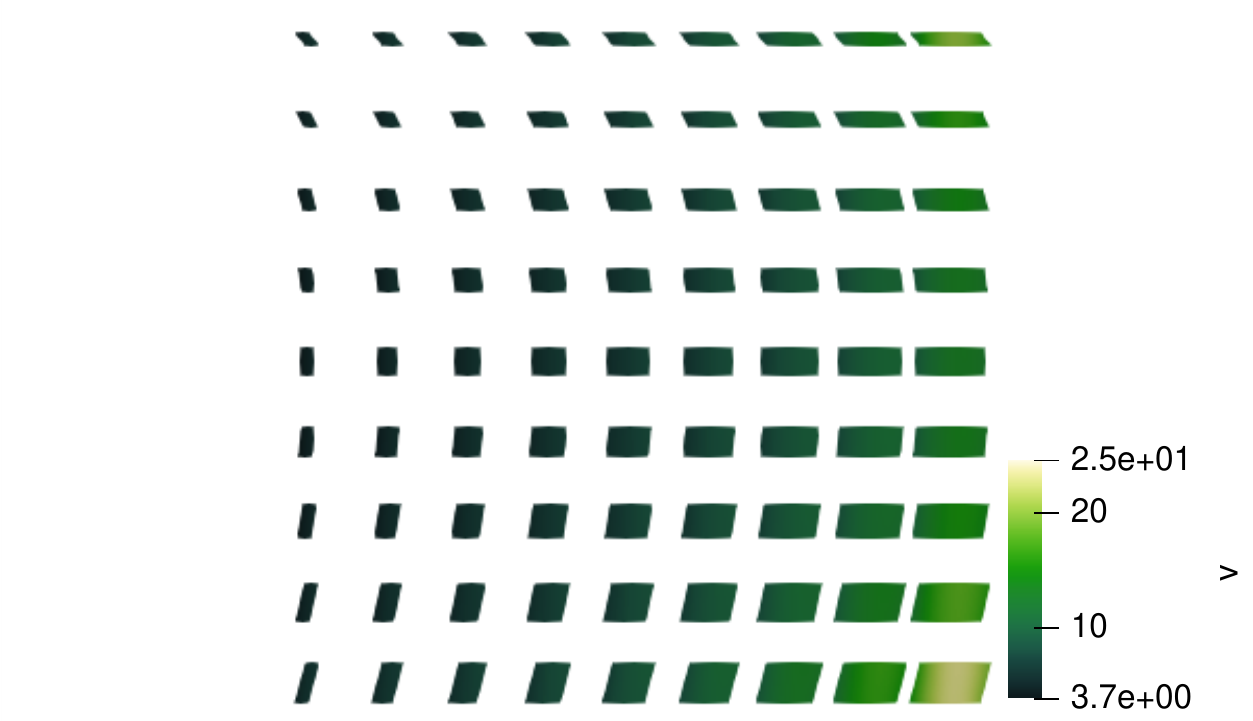}
    \caption{Cell-based nutrient concentration $v(x,y)$, balanced by
      the consumption of $u(x)$ and the production of
      $w(x)$.}\label{fig:v-case-b}
\end{minipage}\end{figure}
\section{Conclusion}
\label{sec:conclusions}
We proposed a parallel finite element treatment of a multiscale model
with distributed heterogeneous microstructures.  To showcase our
framework, we applied it to a system of coupled elliptic PDEs modeling
nutrient transport in plants. After imposing a set of constraints on
the involved parameters and on the regularity of the macroscopic and
microscopic sets so that the underlying coupled system of equations is
weakly solvable, we provided a suitable two-scale finite element
approximation.  Our implementation of this finite element system
provides convergent approximations to the weak solution of the
original system with the expected theoretical order of
convergence. Numerical results demonstrated the influence of the
microstructure geometry on both macroscopic and microscopic solution
profiles. A parallel implementation scales well for the number of
nodes tested.

In practice, the scalability of this setup will not be limited by the
implementation; rather, in many cases there might not be access to
suitable hardware that has a sufficient number of nodes on a single
shared-memory resource.  In order to resolve this, we propose the
following \emph{hybrid} parallelization approach: using a macroscopic
domain decomposition, we can use classic PDE parallelization
techniques to partition the domain on different CPUs, while using the
multithreading method to solve each subdomain in parallel on different
nodes.  This will retain the excellent scaling of the multithreading
approach, while also taking full advantage of the performance of
parallel domain decomposition. Parallel domain decomposition for
single scale problems performs well (see
\cite[e.g.]{yagawa93,klawonn10}) and has been applied in different
multiscale finite element contexts as well (see \cite[e.g.]{hou97}).
The heterogeneous microstructures are represented and implemented as a
continuous map. This allows us to model and compute a family of
microscopic domains without explicitly creating meshes for it.  The
model problem illustrates that, even with relatively simple domain
mappings, one retains a lot of modeling flexibility.

\AM{From the methodological point of view, it is worth noting that
  fixing-boundary techniques are efficient tools regularly used for
  the solvability analysis of free- and moving-boundary problems and
  they have been employed as well for upscaling studies of such
  evolution systems (compare
  \cite{PeterCRAS2006,Wied,eden2024}). Moving sets in high-dimensional
  spaces are notoriously hard to deal with. It is therefore a natural
  question to ask to which extent transformation tools (like the
  successful Hanzawa mapping) can be used as well for both numerical
  computations and error analysis studies of multiscale problems.}

Future improvements of this approach should include ways of refining
computations at the level of the microscopic systems as a function of
error indicators and local geometry.  For instance, one improvement to
this finite element strategy would be to compute a global microscopic
functional: an error indicator defined over all microscopic grids.
This quantity could then be used to refine the reference cell and
maximize globally the error reduction.  Inspired by ideas from
\cite{LeBris}, another approach would be using random diffeomorphisms
instead of deterministic ones.  This would allow us to include, under
specific conditions, distributions of microscopic defects. It even
allows for cases where the choice of the model parameters depends on
stationary random ergodic microstructures.  We plan to address some of
these challenges in a forthcoming work.

\section*{Acknowledgments}
OR would like to gratefully acknowledge partial financial support from
the Institute of Mathematics and its Applications, the Marie
Skłodowska--Curie ModCompShock ITN and thank the University of Sussex
for the kind hospitality. AM thanks SNIC project 2020/9-178 (HPC2N)
\AM{and NAISS 2023/22-1283 }for computational resources and
acknowledges the grant VR 2018-03648 ``\textit{Homogenization and
  dimension reduction of thin heterogeneous layers}'' for partial
financial support.

\bibliographystyle{abbrvnat}

\end{document}